\documentclass[11pt,reqno]{amsart}
\usepackage[a4paper,margin=1.08in]{geometry}
\usepackage[T1]{fontenc}
\usepackage{lmodern}
\usepackage{amssymb,mathtools}
\usepackage{microtype}
\usepackage{booktabs,tabularx}
\usepackage{enumitem}
\usepackage[hidelinks]{hyperref}

\newtheorem{theorem}{Theorem}[section]
\newtheorem{lemma}[theorem]{Lemma}
\newtheorem{corollary}[theorem]{Corollary}
\newtheorem{conjecture}[theorem]{Conjecture}
\theoremstyle{remark}

\numberwithin{equation}{section}

\DeclareMathOperator{\rank}{rank}

\setlist[enumerate]{label=\textup{(\roman*)},leftmargin=2.2em,itemsep=2pt,topsep=4pt}
\hypersetup{pdftitle={On additive bases of finite abelian groups},
  pdfauthor={Dongchun Han and Hanbin Zhang}}

\title[On additive bases of finite abelian groups]{On additive bases of finite abelian groups}

\author{Weidong Gao}
\address{Center for Applied Mathematics, Tianjin University, Center for Applied Mathematics and KL-AAGDM, Tianjin University, Tianjin
300072, P.R. China}
\email{wdgao@nankai.edu.cn}

\author{Hanbin Zhang}
\address{School of Mathematics (Zhuhai), Sun Yat-sen University, Zhuhai 519082, Guangdong, P.R. China}
\email{zhanghb68@mail.sysu.edu.cn}

\subjclass[2020]{Primary 11B75; Secondary 11P70, 20K01}
\keywords{Additive basis, regular sequence, finite abelian group, Kneser's theorem}
\date{}

\begin{document}

\begin{abstract}
Let $G$ be a finite abelian group, let $p$ be the smallest prime divisor of $|G|$, and let
$c_0(G)$ be the least integer $t$ such that every regular sequence over $G$ of length at least $t$ is an
additive basis. Additive bases have been extensively studied in additive combinatorics, and
the exact value of $c_0(G)$ has been determined for many groups of rank at least three and some
groups of rank two (in which case groups of the form $C_p\oplus C_{pn}$ turn out to be exceptional). In
this paper, we continue the study of $c_0(G)$ and obtain the following results.
\begin{enumerate}[label=\textup{(\arabic*)},leftmargin=2em]
\item Together with known results, we prove that $c_0(G)=|G|/p+p-2$ for all abelian groups $G$
with $\rank(G)\geq 3$.
\item For distinct odd primes $p<q$, we determine
\[
c_0(C_p\oplus C_{pq})=pq+2p-3,\qquad
c_0(C_q\oplus C_{pq})=q^2+p-2.
\]
\item Let $p$ be an odd prime, let $2\leq a\leq b$, and let $n$ be a positive integer with no prime divisor
smaller than $p$. Then
\[
c_0(C_{p^a}\oplus C_{p^bn})=p^{a+b-1}n+p-2.
\]
\item For every odd integer $n\geq 3$, we prove
\[
c_0(C_3\oplus C_{3n})=3n+3.
\]
\end{enumerate}
Therefore, $c_0(G)$ is determined except possibly for groups of the form $C_p\oplus C_{pn}$ with $p\geq 5$.
\end{abstract}

\maketitle

\section{Introduction}

Let $G$ be a nontrivial finite abelian group, written additively. Let $S=g_1\cdot\ldots\cdot g_k$ be a sequence over $G$ (unordered,  repetition is allowed, and $g_i\in G$ for $1\le i\le k$) and let $\sigma(S)=\sum^k_{i=1}g_i$. We denote by $T\mid S$ for a subsequence $T$ of $S$. We use
\[
\Sigma(S)=\{\sigma(T):1\neq T\mid S\},\qquad
\Sigma_0(S)=\Sigma(S)\cup\{0\},
\]
where $1$ denotes the empty sequence. The sequence $S$ is an additive basis of $G$ if $\Sigma(S)=G$.
The critical number $c(G)$ is defined as the smallest integer $t$ such that every subset of $G\setminus\{0\}$
of cardinality at least $t$ is an additive basis of $G$. In 1964, Erd\H{o}s and Heilbronn proposed the
problem of determining $c(G)$, and it was completely determined by 2009 through many authors'
efforts; see~\cite{FreezeGaoGeroldinger} and the references therein, and we also refer to \cite{HanRen} for an important related progress. The additive basis problem has also recently been
considered for finite, not necessarily abelian, groups by Qu and Gao~\cite{QuGao}.

For a subgroup $H\leq G$, let $S_H$ be the subsequence consisting of the terms of $S$ which lie in $H$.
We call $S$ a regular sequence if
\[
|S_H|\leq |H|-1\qquad\text{for every proper subgroup }H<G.
\]
Let
\[
c_0(G)=\min\{t:\text{every regular sequence $S$ over $G$ with $|S|\geq t$ satisfies $\Sigma(S)=G$}\}.
\]
Regularity is the natural replacement for the condition that a set is not concentrated in a proper
subgroup. It is also useful for sequences with repeated terms: if a sequence is not regular, then
redundant terms can be removed until one obtains a regular subsequence with the same set $\Sigma_0$;
see Lemma~\ref{lem:2.2}.

The problem of determining $c_0(G)$ was first proposed by Olson. He conjectured that $c_0(C_p\oplus
C_p)=2p-1$, which was proved and further studied by Peng \cite{PengI,PengII} in 1987.

Let $p$ be the smallest prime divisor of $|G|$. Define
\begin{equation}\label{eq:m}
m(G)=
\begin{cases}
|G|,&G\text{ is cyclic},\\
2p-1,&G\cong C_p\oplus C_p,\\
pn+2p-3,&G\cong C_p\oplus C_{pn}\text{ with }n\geq 2,\\
|G|/p+p-2,&\text{otherwise}.
\end{cases}
\end{equation}
The standard examples give $c_0(G)\geq m(G)$; see Lemma~\ref{lem:2.3}. The following conjecture originates
in~\cite{GHHQZ}, with the boundary value proposed in~\cite{QuHan2017}; see also~\cite{GQZ,GLQW}.

\begin{conjecture}\label{conj:1.1}
For every nontrivial finite abelian group $G$,
\[
c_0(G)=m(G).
\]
\end{conjecture}

The third case of~\eqref{eq:m} is exceptional. In the other noncyclic cases the usual construction is
obtained from a subgroup of index $p$ and has length $|G|/p+p-3$. For $G=C_p\oplus C_{pn}$ with $n\geq 2$
there is a regular sequence of length $pn+2p-4$, which is not an additive basis. The equality
\(
|G|/p=\exp(G)=pn
\)
explains why this family requires a separate argument. 

We first record the known results which will be used below. This also gives a precise account
of the status of Conjecture~\ref{conj:1.1} before presenting our works.

Peng proved
\begin{equation}\label{eq:peng}
c_0(C_p\oplus C_p)=2p-1,\qquad
c_0(C_p^r)=p^{r-1}+p-2\quad(r\geq 3)
\end{equation}
for every prime $p$~\cite{PengI,PengII}. Gao and Hamidoune, and later Gao, et al.,
determined the cyclic groups, the groups of even order, all groups of rank at least five, and
all $p$-groups except at that time for the boundary groups $C_p\oplus C_{p^a}$~\cite{GaoHamidoune,GHHQZ}. Qu and Han also
established an inverse theorem for additive bases~\cite{QuHan2016}, and subsequently proved
\begin{equation}\label{eq:primepower}
c_0(C_p\oplus C_{p^a})=p^a+2p-3\quad(a\geq 2)
\end{equation}
\cite{QuHan2017}. Thus Conjecture~\ref{conj:1.1} is known for every finite abelian $p$-group. For related inverse results
on regular sequences over $C_p\oplus C_p$, see Peng, Qu, and Li~\cite{PengQuLi}.

For groups of higher rank, Gao, Qu, and Zhang proved the following additional cases~\cite{GQZ}:
\begin{enumerate}
\item $\rank(G)=4$, apart from the groups $C_3^3\oplus C_n$ with $3\mid n$, $n$ odd, not a power of $3$, and
$|G|<3.72\times 10^7$;
\item $\rank(G)=3$ and either the smallest prime divisor $p$ satisfies $p\geq 11$, or $3\leq p\leq 7$ and
$|G|\geq 3.72\times 10^7$.
\end{enumerate}
The principal known rank-two results are as follows. Let $G=C_{n_1}\oplus C_{n_2}$ with $1<n_1\mid n_2$,
and let $p$ be the smallest prime divisor of $|G|$.
\begin{enumerate}
\item If $n_1\geq 2p$ and $|G|\geq 72p^6$, then
\[
c_0(G)=|G|/p+p-2
\]
by Gao, Li, Qu, and Wang~\cite{GLQW}.
\item If $p\geq 11$ and $n_1\geq p^2$, then the same equality holds; Qu, Li, Wang, and Zhao also describe
the regular nonbases of length $c_0(G)-1$~\cite{QLWZ}.
\item For every prime $q\geq 5$,
\[
c_0(C_3\oplus C_{3q})=3q+3
\]
by Qu and Li~\cite{QuLi}.
\end{enumerate}
Very recently, Yang obtained general upper bounds for $c_0(C_q\oplus C_{nq})$ when the group has odd
order, together with related results for a zero-sum invariant~\cite{Yang}. In particular, if $q$ is prime,
$n\geq 2$, and $p\geq 3$ is the smallest prime divisor of $nq$, then~\cite[Theorem~3.5]{Yang} gives
\[
c_0(C_q\oplus C_{nq})\leq
\max\left\{nq+2q^2-q-1,\ \frac{nq(q+1)}{p}-q+p-2\right\}.
\]

For convenience of the reader, the exact values already available in the literature are summarized
in Table~\ref{tab:known}.

\begin{table}[htbp]
\caption{Published cases of Conjecture~\ref{conj:1.1} used in this paper.}\label{tab:known}
\small
\begin{tabularx}{\textwidth}{@{}p{0.32\textwidth}Xl@{}}
\toprule
Class of groups & Exact value or condition & Reference\\
\midrule
Cyclic groups & $c_0(G)=|G|$ & \cite{GaoHamidoune,GHHQZ}\\
Groups of even order & $c_0(G)=m(G)$ & \cite{GHHQZ}\\
Elementary abelian groups & Values in~\eqref{eq:peng} & \cite{PengI,PengII}\\
All finite abelian $p$-groups & $c_0(G)=m(G)$, including~\eqref{eq:primepower} & \cite{GHHQZ,QuHan2017}\\
Rank at least five & $c_0(G)=m(G)$ & \cite{GHHQZ}\\
Rank three and four & The ranges stated above & \cite{GQZ}\\
Rank two, large order & $n_1\geq 2p$ and $|G|\geq 72p^6$ & \cite{GLQW}\\
Rank two, large first factor & $p\geq 11$ and $n_1\geq p^2$ & \cite{QLWZ}\\
$C_3\oplus C_{3q}$, $q\geq 5$ prime & $c_0(C_3\oplus C_{3q})=3q+3$ & \cite{QuLi}\\
\bottomrule
\end{tabularx}
\end{table}

Our first result removes all remaining groups of rank at least three.
\begin{theorem}\label{thm:1.2}
If $\rank(G)\geq 3$, then
\[
c_0(G)=m(G).
\]
\end{theorem}


We next settle two basic rank-two families.
\begin{theorem}\label{thm:1.3}
Let $p<q$ be distinct odd primes. Then
\[
c_0(C_p\oplus C_{pq})=pq+2p-3.
\]
\end{theorem}

\begin{theorem}\label{thm:1.4}
Let $p<q$ be distinct odd primes. Then
\[
c_0(C_q\oplus C_{pq})=q^2+p-2.
\]
\end{theorem}


The next theorem deals with groups with first invariant factor divisible by $p^2$ and having no other prime
divisor.
\begin{theorem}\label{thm:1.5}
Let $p$ be an odd prime, let $2\leq a\leq b$, and let $n$ be a positive integer with no prime
divisor smaller than $p$. Then
\[
c_0(C_{p^a}\oplus C_{p^bn})=p^{a+b-1}n+p-2.
\]
\end{theorem}

Combining these results gives a complete answer outside the boundary family.
\begin{corollary}\label{cor:1.6}
Let $G$ be a nontrivial finite abelian group and let $p$ be the smallest prime divisor of $|G|$. If
$G$ is not isomorphic to $C_p\oplus C_{pn}$ for any integer $n\geq 2$, then
\[
c_0(G)=m(G).
\]
\end{corollary}

For the boundary family we obtain a complete result when the smallest prime is $3$.
\begin{theorem}\label{thm:1.7}
For every odd integer $n\geq 3$,
\[
c_0(C_3\oplus C_{3n})=3n+3.
\]
\end{theorem}

\begin{corollary}\label{cor:1.8}
If the smallest prime divisor of $|G|$ is $3$, then
\[
c_0(G)=m(G).
\]
\end{corollary}

In Section~2 we present some preliminaries and auxiliary results. We provide a reduction argument in Section~3, and the
theorem for groups of rank at least three is proved in Section~4. Section~5 gives the estimates
for cosets of a cyclic subgroup and proves Theorems~\ref{thm:1.3} and~\ref{thm:1.4}.
The lifting argument in Section~6 proves Theorem~\ref{thm:1.5} and Corollary~\ref{cor:1.6}.
The groups $C_p\oplus C_{pn}$ are considered in Section~7, and the proof for $p=3$ is given in Section~8.

\section{Preliminaries}\label{sec:preliminaries}

We use the notation introduced in Section~1. For $g\in G$ and $k\geq 0$,
let $g^{[k]}$ denote the sequence consisting of $k$ copies of $g$. If
$T\mid S$, then $ST^{-1}$ is obtained by deleting the terms of $T$, with
their multiplicities. The support of $S$ is denoted by
$\operatorname{supp}(S)$, and $\langle A\rangle$ denotes the subgroup
generated by $A\subseteq G$. For a nonempty subset $A$ of $G$, put
\[
  \operatorname{st}(A)=\{g\in G:g+A=A\}.
\]
We say that $A$ has trivial stabilizer if $\operatorname{st}(A)=\{0\}$.
We use
\[
 \Sigma_0(ST)=\Sigma_0(S)+\Sigma_0(T),\qquad
 \operatorname{st}(\Sigma_0(S))\subseteq\operatorname{st}(\Sigma_0(ST)).
\]

We first recall Kneser's theorem \cite{Kneser}; see also \cite{GeroldingerHalterKoch2006,Grynkiewicz,Nathanson} for detailed discussion.

\begin{theorem}[Kneser]\label{thm:2.1}
Let $A_1,\ldots,A_t$ be finite nonempty subsets of an abelian group, and
let $H=\operatorname{st}(A_1+\cdots+A_t)$. Then
\[
 |A_1+\cdots+A_t|
 \geq \sum_{i=1}^t|A_i+H|-(t-1)|H|.
\]
In particular, if $H=\{0\}$, then
\[
 |A_1+\cdots+A_t|\geq 1+\sum_{i=1}^t(|A_i|-1).
\]
\end{theorem}

The Davenport constant $\mathsf D(G)$ is the least positive integer $d$
such that every sequence over $G$ of length at least $d$ has a nonempty
zero-sum subsequence; see \cite{HanZhang} for a recent study and references therein. We shall use
\begin{equation}\label{eq:Davenport}
 \mathsf D(G)\leq m(G),\qquad
 \mathsf D(C_m\oplus C_n)=m+n-1\quad(1\leq m\mid n);
\end{equation}
see \cite[Lemmas~2.6 and~2.7]{GHHQZ} and
\cite[Lemma~2.5]{GQZ}. We also use $\mathsf D(G)\leq |G|$.
For the trivial group we take $c_0(\{0\})=\mathsf D(\{0\})=1$.

The following three results are already known.

\begin{lemma}[{\cite[Lemma~2.2]{GQZ}}]\label{lem:2.2}
If a sequence $S$ over $G$ is not regular, then it has a proper regular
subsequence $T$ such that $\Sigma_0(T)=\Sigma_0(S)$.
\end{lemma}

\begin{lemma}[{\cite[Lemma~2.3]{GQZ}}]\label{lem:2.3}
For every nontrivial finite abelian group $G$, we have $c_0(G)\geq m(G)$.
\end{lemma}

\begin{lemma}[{\cite[Lemma~2.3]{GHHQZ}}]\label{lem:2.4}
Let $p$ be the smallest prime divisor of $|G|$, and let $S$ be a regular
sequence with
\[
 |S|\geq\max\{|G|/p+p-2,\mathsf D(G)\}.
\]
If $\Sigma(S)\neq G$, then $0\in\Sigma(S)$ and, for every nonempty
subsequence $T\mid S$,
\[
 \operatorname{st}(\Sigma_0(T))=\{0\},\qquad
 |\Sigma_0(T)|\geq |T|+1.
\]
\end{lemma}

To keep the later estimates short, we denote
\begin{equation}\label{eq:excess}
 e(T)=|\Sigma_0(T)|-|T|-1.
\end{equation}
Thus $e(1)=0$, and Lemma~\ref{lem:2.4} gives $e(T)\geq 0$ for every
$T\mid S$ under its hypotheses.

\begin{lemma}\label{lem:2.5}
Let $T$ be a sequence over $G$ and $g\in G\setminus\{0\}$. If
$\Sigma_0(T)$ has trivial stabilizer, then $e(Tg)\geq e(T)$.
\end{lemma}

\begin{proof}
Since $g\notin\operatorname{st}(\Sigma_0(T))$, the set
$\Sigma_0(T)\cup(g+\Sigma_0(T))$ has at least one more element than
$\Sigma_0(T)$.
\end{proof}

We shall also use the following consequence of the Cauchy--Davenport
theorem \cite{Davenport1935}.

\begin{lemma}\label{lem:2.6}
Let $p$ be prime and let $U$ be a sequence of nonzero elements of $C_p$.
Then
\[
 |\Sigma_0(U)|\geq\min\{p,|U|+1\}.
\]
In particular, $\Sigma_0(U)=C_p$ if $|U|\geq p-1$.
\end{lemma}

\begin{lemma}\label{lem:2.7}
Let $H\leq G$, let $A$ be a sequence over $H$, and let $T$ be a sequence
over $G$. If $\varphi:G\longrightarrow G/H$ is the natural homomorphism,
then
\[
 |\Sigma_0(AT)|\geq |\Sigma_0(A)|\,|\Sigma_0(\varphi(T))|.
\]
\end{lemma}

\begin{proof}
For each $\gamma\in\Sigma_0(\varphi(T))$, choose $T_\gamma\mid T$ with
$\varphi(\sigma(T_\gamma))=\gamma$. The sets
$\sigma(T_\gamma)+\Sigma_0(A)$ lie in distinct cosets of $H$ and are
contained in $\Sigma_0(AT)$.
\end{proof}

\begin{lemma}[{\cite[Lemma~2.4]{GHHQZ}}]\label{lem:2.8}
For every finite abelian group $A$, we have $c_0(A)\leq |A|$.
\end{lemma}

\begin{lemma}\label{lem:2.9}
Let $W$ be a sequence of nonzero elements of a finite abelian group $A$.
If $\Sigma_0(W)$ has trivial stabilizer, then $W$ is regular over $A$.
\end{lemma}

\begin{proof}
Otherwise choose a proper subgroup $J<A$ of least order with
$|W_J|\geq |J|$. Then $J\neq\{0\}$ and $W_J$ is regular over $J$.
Lemma~\ref{lem:2.8} gives $\Sigma_0(W_J)=J$, so
$\Sigma_0(W)=J+\Sigma_0(WW_J^{-1})$ is stabilized by $J$, a contradiction.
\end{proof}

The next lemma allows us to retain a prescribed regular subsequence when
choosing a short subsequence with all possible subset sums.

\begin{lemma}\label{lem:2.10}
Let $X$ be a sequence over a finite abelian group $A$ with
$\Sigma_0(X)=A$. Suppose that $Z\mid X$ is regular over $A$ and
$|Z|\leq c_0(A)-1$. Then there is a subsequence $R\mid X$ such that
\[
 Z\mid R,\qquad |R|\leq c_0(A),\qquad \Sigma_0(R)=A.
\]
\end{lemma}

\begin{proof}
We delete terms outside $Z$, preserving the set of subset sums, until
the remaining sequence is regular. To make this more precise, let $Y$ be the
current sequence. If it is not regular, choose $J<A$ with
$|Y_J|\geq |J|$, and write
\[
 Y_J=Z_Jy_1\cdots y_m,\qquad t=|Z_J|\leq |J|-1.
\]
We first show that $|\Sigma_0(Z_J)|\geq t+1$. Put
$H=\operatorname{st}(\Sigma_0(Z_J))$ and $h=|H|$. Since
$H\subseteq\Sigma_0(Z_J)$, the assertion is immediate if $h\geq t+1$.
Otherwise $H<J$, and regularity gives $|Z_H|\leq h-1$. Theorem \ref{thm:2.1}
therefore yields
\[
 |\Sigma_0(Z_J)|\geq (t+2-h)h\geq t+1,
\]
where the last inequality is equivalent to
$(h-1)(t-h+1)\geq 0$.

Starting from $\Sigma_0(Z_J)$, adjoin $y_1,\ldots,y_m$ successively.
If every term enlarged the sumset, the final set would have at least
$t+1+m=|Y_J|+1>|J|$ elements. Hence one of these terms can be deleted
without changing $\Sigma_0(Y_J)$, and therefore without changing
$\Sigma_0(Y)$.

The process ends with a regular subsequence $Y'\mid X$ containing $Z$
and satisfying $\Sigma_0(Y')=A$. If $|Y'|\leq c_0(A)$, take $R=Y'$.
Otherwise take any $c_0(A)$ terms of $Y'$ that include $Z$; their sequence
is regular and is an additive basis of $A$.
\end{proof}

We need the following two facts concerning $C_p^2$. The equality for
$c_0(C_p^2)$ is due to Peng \cite[Theorem~2]{PengI}; the additional
estimate follows from \cite[Theorem~3.12]{GPZ2013} (see also
\cite[Lemma~2.4]{QuLi}).

\begin{theorem}\label{thm:2.11}
Let $p$ be prime. Then $c_0(C_p^2)=2p-1$. Moreover, if $Y$ is regular
over $C_p^2$ and $|Y|=2p-2$, then $|\Sigma_0(Y)|\geq p^2-1$.
\end{theorem}

\begin{lemma}\label{lem:2.12}
If $Y$ is a sequence over $C_p^2$ with $\Sigma_0(Y)=C_p^2$, then it has
a subsequence $C$ with $\Sigma_0(C)=C_p^2$ such that either
$|C|\leq 2p-2$, or $|C|=2p-1$ and $C$ is regular.
\end{lemma}

\begin{proof}
By Lemma~\ref{lem:2.2}, we may replace $Y$ by a regular subsequence with
the same subset sums. If its length exceeds $2p-2$, take any $2p-1$ of
its terms and apply Theorem~\ref{thm:2.11}.
\end{proof}

\section{A reduction argument}\label{sec:quotients}

Let $q$ be prime and let $\pi:G\longrightarrow C_q^2$ be surjective.
For every subgroup $M\leq C_q^2$ of order $q$, let $H_M=\pi^{-1}(M)$.
We first recall a consequence of Hall's marriage theorem
\cite{Hall}. In the form needed here, that theorem states that a finite
family of sets $E_1,\ldots,E_t$ has distinct representatives
$x_i\in E_i$ if and only if
\[
 \left|\bigcup_{i\in I}E_i\right|\geq |I|
 \quad\text{for every }I\subseteq\{1,\ldots,t\}.
\]
When this theorem is applied to a sequence, different occurrences of the
same group element are regarded as distinct choices.

\begin{lemma}\label{lem:3.1}
Let $\mathcal L$ be a collection of subgroups of $C_q^2$ of order $q$,
and let $\ell\geq 1$. Suppose that, for every nonempty
$\mathcal J\subseteq\mathcal L$, at least $\ell|\mathcal J|$ terms of a
sequence $S$ lie outside at least one of the subgroups $H_M$,
$M\in\mathcal J$. Then there are pairwise disjoint subsequences
$T_M\mid S$, $M\in\mathcal L$, such that $|T_M|=\ell$ and all terms of
$T_M$ lie outside $H_M$.
\end{lemma}

\begin{proof}
For each $M\in\mathcal L$, take $\ell$ copies of the set of occurrences
in $S$ outside $H_M$. A collection of copies involving precisely the
members of $\mathcal J$ has at most $\ell|\mathcal J|$ members, while
their union has at least this many occurrences by hypothesis. Hall's
theorem gives the required choices.
\end{proof}

Now we present the setting we use throughout this section. Suppose that $p<q$ are odd primes, that $p$ is the smallest prime
divisor of $|G|$, and that $G$ has a quotient $C_q^2$. Write
\[
 \pi:G\longrightarrow C_q^2,\qquad K=\ker\pi,\qquad
 |G|=pq^2h,\qquad |K|=ph.
\]
Assume that $m(G)=|G|/p+p-2$, and put
\[
 L=q^2h+p-2,\qquad N=|G|-1-L=(p-1)q^2h-p+1.
\]

\begin{theorem}\label{thm:3.2}
In the preceding setting, suppose that integers $r,\ell$ satisfy
$1\leq r\leq q+1$, $1\leq\ell\leq q-1$, and
\begin{align}
 \frac{r\ell L}{q+1}&\geq N,\label{eq:3.1}\\
 \ell(L-r\ell)&\geq N,\label{eq:3.2}\\
 h(q^2-p)+p-1&\geq r\ell.\label{eq:3.3}
\end{align}
Then $c_0(G)=L$.
\end{theorem}

\begin{proof}
By Lemma~\ref{lem:2.3}, it suffices to prove the upper bound. Suppose
that $S$ is regular, $|S|=L$, and $\Sigma(S)\neq G$. By
\eqref{eq:Davenport} and Lemma~\ref{lem:2.4}, every $\Sigma_0(T)$,
$T\mid S$, has trivial stabilizer.

For each subgroup $M\leq C_q^2$ of order $q$, let $P_M$ consist of the
terms of $S$ whose images lie in $M\setminus\{0\}$. Put
$d_M=|P_M|$ and $z=|S_K|$. Then
\[
 \sum_Md_M=L-z.
\]
Regularity in $H_M$ gives $z+d_M\leq pqh-1$, and hence
\begin{equation}\label{eq:3.4}
 |SS_{H_M}^{-1}|\geq qh(q-p)+p-1\geq q-1.
\end{equation}
Choose $r$ subgroups $M_1,\ldots,M_r$ for which the sum of the $d_M$
is largest. Thus
\begin{equation}\label{eq:3.5}
 \sum_{i=1}^rd_{M_i}\geq\frac{r(L-z)}{q+1}.
\end{equation}

We apply Lemma~\ref{lem:3.1} to these $r$ subgroups. For a single
subgroup, \eqref{eq:3.4} provides at least $\ell$ allowable terms.
For two or more distinct subgroups, their intersection is $\{0\}$, so
every term outside $K$ lies outside at least one of their inverse images.
Moreover,
\[
 |S|-|S_K|\geq L-(ph-1)=h(q^2-p)+p-1\geq r\ell.
\]
We obtain disjoint subsequences $T_1,\ldots,T_r$, each of length $\ell$,
whose terms lie outside $H_{M_i}$, respectively.

We may make this choice with the following additional property. If the
subsequences $P_M$ with $M\notin\{M_1,\ldots,M_r\}$ contain at least
$r\ell$ terms altogether, choose all the $T_i$ from those terms.
If they contain fewer than $r\ell$ terms, choose the $T_i$ so that all
of them are used. Indeed, starting from any choice supplied by Hall's
theorem, an unused term from an unselected subgroup can replace a
selected term from one of $P_{M_1},\ldots,P_{M_r}$: the new term lies
outside every $H_{M_i}$.

After removing the $T_i$, let $A_i$ contain the remaining terms of
$P_{M_i}$, and distribute the terms of $S_K$ among the $A_i$. Each
$A_i$ is supported in $H_{M_i}$. Put $a=\sum_{i=1}^r|A_i|$. In the
first case, \eqref{eq:3.5} gives
\[
 a\geq z+\frac{r(L-z)}{q+1}\geq\frac{rL}{q+1}.
\]
In the second case, every remaining term belongs to an $A_i$, so
$a=L-r\ell$. Thus \eqref{eq:3.1} and \eqref{eq:3.2} give
$\ell a\geq N$ in both cases.

The image of $T_i$ in $G/H_{M_i}\cong C_q$ consists of $\ell$ nonzero
terms. Lemmas~\ref{lem:2.4}, \ref{lem:2.6}, and \ref{lem:2.7} yield
\[
 |\Sigma_0(A_iT_i)|\geq(\ell+1)|\Sigma_0(A_i)|
 \geq(\ell+1)(|A_i|+1),
\]
and therefore
\[
 |\Sigma_0(A_iT_i)|-1\geq |A_iT_i|+\ell|A_i|.
\]
Now, we apply Theorem \ref{thm:2.1}.
Since $\Sigma_0(S)$ has trivial stabilizer, we obtain
\[
 |\Sigma_0(S)|\geq 1+|S|+\ell a
 \geq 1+L+N=|G|,
\]
a contradiction.
\end{proof}

\begin{theorem}\label{thm:3.3}
In the preceding setting, if either $q\geq 2p-1$ or $h\geq p$, then
\[
 c_0(G)=|G|/p+p-2.
\]
\end{theorem}

\begin{proof}
If $q\geq 2p-1$, take $r=p$ and $\ell=q-1$ in
Theorem~\ref{thm:3.2}. It is not hard to verify
\eqref{eq:3.1}--\eqref{eq:3.3}. For the second inequality, one may use
\[
 q^2(q-p)-p(q-1)^2\geq\frac{q-1}{2},
\]
which follows from $p\leq(q+1)/2$.

Suppose next that $p<q<2p-1$ and $h\geq p$. Take
$r=q+1$ and $\ell=q-1$. The first and third inequalities in
Theorem~\ref{thm:3.2} follow immediately by substitution. For the second,
\[
 (q-1)\bigl(L-(q+1)(q-1)\bigr)-N
 =q\bigl(hq(q-p)-q^2+q+p-1\bigr)>0,
\]
because $h(q-p)\geq 2p>q$.
\end{proof}

We shall also need a bound obtained from a quotient associated with the
smallest prime divisor.

\begin{theorem}\label{thm:3.4}
Let $A$ be a finite abelian group of odd order, and let $p$ be the
smallest prime divisor of $|A|$. If $A$ has a quotient $C_p^2$, then
\[
 c_0(A)\leq |A|/p+p^2-p-1.
\]
\end{theorem}

\begin{proof}
Let $\pi:A\longrightarrow C_p^2$ be surjective, put $K=\ker\pi$, and
let $L=p|K|+p^2-p-1$. By \eqref{eq:m}, $L\geq m(A)$. Suppose that
$S$ is regular, $|S|=L$, and $\Sigma(S)\neq A$. Lemma~\ref{lem:2.4}
applies to every subsequence of $S$.

For each subgroup $M\leq C_p^2$ of order $p$, put
$H_M=\pi^{-1}(M)$. Regularity gives at least
\[
 L-(p|K|-1)=p(p-1)
\]
terms outside $H_M$. For two or more distinct $M$, every term outside
$K$ lies outside at least one of the corresponding $H_M$, and
\[
 |S|-|S_K|\geq L-(|K|-1)
 =(p-1)(|K|+p)\geq(p+1)(p-1).
\]
Lemma~\ref{lem:3.1} therefore gives disjoint subsequences $T_M$, one
for each of the $p+1$ subgroups $M$, such that $|T_M|=p-1$ and all
terms of $T_M$ lie outside $H_M$.

Assign each remaining term outside $K$ to the sequence $A_M$
corresponding to the subgroup containing its image, and distribute the
terms of $S_K$ among these sequences. Then $A_M$ is supported in $H_M$
and
\[
 \sum_M|A_M|=L-(p^2-1)=p|K|-p.
\]
The subset sums of the image of $T_M$ fill $A/H_M\cong C_p$.
Consequently, Lemmas~\ref{lem:2.4}, \ref{lem:2.6}, and \ref{lem:2.7}
give
\[
 |\Sigma_0(A_MT_M)|-1
 \geq |A_MT_M|+(p-1)|A_M|.
\]
Kneser's theorem now yields
\[
 |\Sigma_0(S)|\geq 1+L+(p-1)(p|K|-p)=p^2|K|=|A|,
\]
a contradiction.
\end{proof}

\section{Groups of rank at least three}\label{sec:higher-rank}

In this section, we complete all the remaining cases for groups of rank at least three. We first bound the number of terms in the kernel of a homomorphism
onto $C_p^2$.

\begin{theorem}\label{thm:4.1}
Let $p$ be the smallest prime divisor of $|G|$, and suppose that
$m(G)=|G|/p+p-2$. Let $\pi:G\longrightarrow C_p^2$ be surjective,
with kernel $K$ satisfying $|K|\geq p^2$. If $S$ is regular,
$|S|=p|K|+p-2$, and $\Sigma(S)\neq G$, then $|S_K|\leq p-2$.
\end{theorem}

\begin{proof}
Let $L=p|K|+p-2$. For each subgroup $M\leq C_p^2$ of order $p$, let
$H_M=\pi^{-1}(M)$. Regularity gives at least $p-1$ terms outside
$H_M$. For two or more distinct subgroups, every term outside $K$
lies outside at least one of their inverse images, and
\[
 |S|-|S_K|\geq L-(|K|-1)
 =(p-1)(|K|+1)\geq(p+1)(p-1).
\]
By Lemma~\ref{lem:3.1}, there are disjoint subsequences $T_M$, one
for each $M$, of length $p-1$ and supported outside $H_M$.

Remove these $p^2-1$ terms. Assign every remaining term outside $K$
to the sequence $A_M$ determined by its image, and distribute the
terms of $S_K$ among the $A_M$. Then
\[
 \sum_M|A_M|=L-p^2+1.
\]
Lemmas~\ref{lem:2.6} and \ref{lem:2.7} give
$|\Sigma_0(A_MT_M)|\geq p|\Sigma_0(A_M)|$, whence
\[
 |\Sigma_0(A_MT_M)|-1
 \geq |A_MT_M|+(p-1)|A_M|+pe(A_M).
\]
By Lemma~\ref{lem:2.4} and Theorem \ref{thm:2.1},
\begin{align*}
 |\Sigma_0(S)|
 &\geq 1+L+(p-1)(L-p^2+1)+p\sum_Me(A_M)\\
 &=|G|-p(p-1)^2+p\sum_Me(A_M).
\end{align*}
Since $|\Sigma_0(S)|\leq |G|-1$, this implies
\begin{equation}\label{eq:4.1}
 \sum_Me(A_M)\leq(p-1)^2-1.
\end{equation}
This inequality holds for every distribution of the terms of $S_K$.

Suppose that $|S_K|\geq p-1$. After the $T_M$ have been removed, the
number of remaining terms outside $K$ is at least
\[
 (p-1)(|K|+1)-(p^2-1)
 =(p-1)(|K|-p)\geq p(p-1)^2.
\]
Since $p(p-1)^2>(p+1)(p-2)$, some $A_M$ contains at least $p-1$
terms outside $K$. Assign $p-1$ terms of $S_K$ to that same $A_M$.
It then contains disjoint subsequences $U,V$ of length $p-1$, with
$U$ supported in $K$ and $V$ supported outside $K$.
Lemma~\ref{lem:2.4} gives $|\Sigma_0(U)|\geq p$, and the subset sums of
the image of $V$ fill $H_M/K\cong C_p$. Hence
\[
 |\Sigma_0(UV)|\geq p^2,\qquad
 e(UV)\geq p^2-(2p-2)-1=(p-1)^2.
\]
Repeated use of Lemma~\ref{lem:2.5} gives
$e(A_M)\geq(p-1)^2$, contrary to \eqref{eq:4.1}.
\end{proof}

\begin{theorem}\label{thm:4.2}
Let $G$ be a finite abelian group of odd order, and let $p$ be the
smallest prime divisor of $|G|$. If the Sylow $p$-subgroup of $G$ has
rank at least three, then $c_0(G)=m(G)$.
\end{theorem}

\begin{proof}
The result is known for elementary abelian groups \cite{PengII}.
Assume that $G$ is not elementary abelian, and write
$V=G/pG\cong C_p^t$, where $t\geq 3$ is the rank of the Sylow
$p$-subgroup. Suppose that $S$ is regular and
\[
 |S|=|G|/p+p-2,\qquad \Sigma(S)\neq G.
\]
For every subgroup $W\leq V$ of order $p^{t-2}$, let $K_W$ be its
inverse image in $G$. Then $G/K_W\cong C_p^2$. Since $G$ is not
elementary abelian, either its Sylow $p$-subgroup has order at least
$p^{t+1}$, or $|G|$ has a prime divisor larger than $p$. Thus
$|G|\geq p^{t+1}$ and
\[
 |K_W|=|G|/p^2\geq p^{t-1}\geq p^2.
\]
Theorem~\ref{thm:4.1} gives $|S_{K_W}|\leq p-2$ for every $W$.

We average these inequalities over all subgroups $W$ of order
$p^{t-2}$. Every nonzero element of $V$ belongs to the same number
of such subgroups. Counting the pairs $(x,W)$ with $0\neq x\in W$
shows that the proportion containing any fixed nonzero element is
\[
 \frac{p^{t-2}-1}{p^t-1}.
\]
If $z$ terms of $S$ have zero image in $V$, the average of
$|S_{K_W}|$ is
\[
 z+(|S|-z)\frac{p^{t-2}-1}{p^t-1}
 \geq |S|\frac{p^{t-2}-1}{p^t-1}.
\]
Consequently,
\[
 |S|\leq(p-2)\frac{p^t-1}{p^{t-2}-1}
 =(p-2)\left(p^2+\frac{p^2-1}{p^{t-2}-1}\right)
 \leq(p-2)(p^2+p+1)<p^3.
\]
On the other hand, $|S|=|G|/p+p-2\geq p^t+p-2\geq p^3$, a
contradiction. Lemma~\ref{lem:2.3} supplies the lower bound.
\end{proof}

\begin{proof}[Proof of Theorem~\ref{thm:1.2}]
The result is known when $|G|$ is even \cite{GHHQZ}, so assume that
$|G|$ is odd. Let $p$ be its smallest prime divisor. If the Sylow
$p$-subgroup has rank at least three, apply Theorem~\ref{thm:4.2}.
Otherwise write
\[
 G=C_{n_1}\oplus\cdots\oplus C_{n_r},\qquad
 1<n_1\mid\cdots\mid n_r,\qquad r\geq 3.
\]
If $p\mid n_1$, the Sylow $p$-subgroup would have rank at least three.
Thus a prime divisor $q$ of $n_1$ satisfies $q>p$. The group $G$ has
a quotient $C_q^2$, and $q^3\mid |G|$. Writing $|G|=pq^2h$, we have
$q\mid h$ and hence $h\geq q>p$. The desired result follows from Theorem~\ref{thm:3.3}.
\end{proof}

\section{Proofs of Theorems~\ref{thm:1.3} and~\ref{thm:1.4}}

We begin with two elementary estimates. In a direct sum $X\oplus C_Q$, we identify $C_Q$ with the subgroup $\{0\}\oplus C_Q$.

\begin{lemma}\label{lem:5.1}
Let $Q=q^s$, where $q$ is prime and $s\geq 1$, and let $X$ be a finite abelian group with $q\nmid |X|$. Suppose that $A\subseteq X\oplus C_Q$ projects onto $X$ and contains no coset of $C_Q$. If $g=(x,y)$ and $y$ generates $C_Q$, then
\[
 |A\cup(A+g)|\geq |A|+[X:\langle x\rangle]
 \geq |A|+\frac{|X|}{\exp(X)}.
\]
\end{lemma}

\begin{proof}
Let $C$ be a coset of $\langle x\rangle$ in $X$. Since $A$ projects onto $X$, the set $A\cap(C\times C_Q)$ is nonempty. If adjoining its translate by $g$ added no element, this set would be invariant under $g$. But $\operatorname{ord}(x)$ and $Q$ are coprime, so
\[
 \langle g\rangle=\langle x\rangle\oplus C_Q.
\]
It would follow that $C\times C_Q\subseteq A$, a contradiction. Thus at least one element is added in each of the $[X:\langle x\rangle]$ disjoint sets $C\times C_Q$.
\end{proof}

\begin{lemma}\label{lem:5.2}
Let $p\ne q$ be primes. Let $A$ be a sequence of length $2p-2$ over $C_p^2\oplus C_q$ whose projection to $C_p^2$ is regular. If $g=(x,y)$ with $x\ne0$ and $y\ne0$, then $\Sigma_0(Ag)$ projects onto $C_p^2$ and
\[
 |\Sigma_0(Ag)|\geq p^2+p-1.
\]
\end{lemma}

\begin{proof}
Let $B$ be the projection of $\Sigma_0(A)$ to $C_p^2$. By Theorem~\ref{thm:2.11}, $|B|\geq p^2-1$. Since $x\ne0$, we have $B\cup(x+B)=C_p^2$.

Let $L=\langle x\rangle$. If a coset $C$ of $L$ is contained in $B$, then $\Sigma_0(A)\cap(C\times C_q)$ has at least $p$ elements. Its union with its translate by $g$ has at least $p+1$ elements. Indeed, otherwise it would have exactly $p$ elements and would be invariant under $g$. Since $g$ generates $L\oplus C_q$, it would then contain all $pq$ elements of $C\times C_q$, a contradiction.

If $B=C_p^2$, this argument applies to all $p$ cosets of $L$. Otherwise it applies to $p-1$ cosets. The remaining coset contributes at least $p$ elements after $g$ is adjoined, since the projection is now all of $C_p^2$. Hence in either case
\[
 |\Sigma_0(Ag)|\geq (p-1)(p+1)+p=p^2+p-1.
\]
\end{proof}

Throughout the rest of this section, $p<q$ are distinct odd primes. First let
\[
 G=C_p\oplus C_{pq}\cong P\oplus Q,
 \qquad P\cong C_p^2,\quad Q\cong C_q,
\]
and let $\rho:G\to P$ be the projection. By Lemma~\ref{lem:2.3}, it remains to prove the upper bound in Theorem~\ref{thm:1.3}. Suppose that $S$ is regular and
\[
 |S|=pq+2p-3,\qquad \Sigma(S)\ne G.
\]
Since $\mathsf D(G)=p+pq-1\leq |S|$, we can use Lemma~\ref{lem:2.4}.

\begin{lemma}\label{lem:5.3}
There is a factorization $S=RT$ such that
\[
 |R|=2p-1,\qquad |T|=pq-2,
\]
$\Sigma_0(\rho(R))=P$, every term of $T$ has nonzero $Q$-coordinate, and
\[
 |\Sigma_0(R)|\geq p^2+p-1.
\]
\end{lemma}

\begin{proof}
The sequence $S_P$ is regular over $P$. If $|S_P|\geq2p-1$, then Theorem~\ref{thm:2.11} gives $\Sigma_0(S_P)=P$, contrary to Lemma~\ref{lem:2.4}. Hence $|S_P|\leq2p-2$.

For each subgroup $L\leq P$ of order $p$, let $a_L$ count the terms of $S$ whose projection belongs to $L\setminus\{0\}$. Since $|L\oplus Q|=pq$, regularity gives
\begin{equation}\label{eq:5.1}
 \sum_{M\ne L}a_M\geq pq+2p-3-(pq-1)=2p-2.
\end{equation}
Consequently,
\begin{equation}\label{eq:5.2}
 \sum_L\min\{a_L,p-1\}\geq2p-2.
\end{equation}
Indeed, this is immediate if two of the $a_L$ are at least $p-1$. Otherwise choose $L$ with $a_L$ largest and apply \eqref{eq:5.1}; none of the other counts is reduced by taking the minimum.

Let $c_L$ count the terms of $S_P$ in $L$. Then $c_L\leq p-1$, $c_L\leq a_L$, and $\sum_Lc_L\leq2p-2$. Thus \eqref{eq:5.2} allows us to choose $R_0\mid S$ of length $2p-2$ containing $S_P$, with no zero projected term and at most $p-1$ projected terms in each subgroup of order $p$. In particular, $\rho(R_0)$ is regular over $P$.

Every term of $SR_0^{-1}$ has nonzero $Q$-coordinate. This sequence has length $pq-1$, whereas at most $q-1$ of its terms lie in $Q$. We can therefore choose $g\mid SR_0^{-1}$ with both coordinates nonzero. Put $R=R_0g$ and $T=SR^{-1}$, and apply Lemma~\ref{lem:5.2}.
\end{proof}

\begin{lemma}\label{lem:5.4}
For the sequence $T$ in Lemma~\ref{lem:5.3}, either $\Sigma_0(T)$ contains a coset of $Q$, or
\[
 \Sigma_0(\rho(T))=P
 \quad\text{and}\quad
 |\Sigma_0(T)|\geq p^2(q-1)-1.
\]
\end{lemma}

\begin{proof}
Suppose that $\Sigma_0(T)$ contains no coset of $Q$. Let $B=\Sigma_0(\rho(T))$ and $H=\operatorname{st}_P(B)$. We first show that $B=P$.

Suppose that $H=\{0\}$. Let $Z=T_Q$, and put $z=|Z|$ and $m=|T|-z$. Kneser's theorem gives $|B|\geq m+1$. If $B\ne P$, then $m\leq p^2-2$. Also, $\Sigma_0(Z)\subseteq Q$, so Lemma~\ref{lem:2.4} and our assumption give $z+1\leq q-1$; this also holds when $Z$ is empty. Therefore
\[
 pq-2=|T|\leq p^2+q-4,
\]
which is impossible since $q\geq p+2$.

Suppose next that $|H|=p$. Write $T=IU$, where $I$ consists of the terms whose projections lie in $H$. Since the image of $B$ in $P/H$ is proper, Lemma~\ref{lem:2.6} gives $|U|\leq p-2$. Moreover, $\Sigma_0(I)\subseteq H\oplus Q$, and each of the $p$ cosets of $Q$ in this subgroup contains at most $q-1$ elements of $\Sigma_0(I)$. Thus
\[
 |I|+1\leq |\Sigma_0(I)|\leq p(q-1),
\]
whence $|T|\leq pq-3$, another contradiction. It follows that $H=P$, and hence $B=P$.

By Lemma~\ref{lem:2.12}, choose $C\mid T$ with $\Sigma_0(\rho(C))=P$ such that either $|C|\leq2p-2$, or $|C|=2p-1$ and $\rho(C)$ is regular. In the first case, $|\Sigma_0(C)|\geq p^2$. Adjoin the remaining terms of $T$ one at a time. Each has nonzero $Q$-coordinate, and every intermediate sumset projects onto $P$ and contains no coset of $Q$. Lemma~\ref{lem:5.1} gives
\[
 |\Sigma_0(T)|\geq p^2+p\bigl(pq-2-(2p-2)\bigr)=p^2(q-1).
\]
In the second case, remove any term $g$ of $C$. The sequence $\rho(Cg^{-1})$ is regular of length $2p-2$, and both coordinates of $g$ are nonzero. Lemma~\ref{lem:5.2} gives $|\Sigma_0(C)|\geq p^2+p-1$. Applying Lemma~\ref{lem:5.1} to the remaining terms yields
\[
 |\Sigma_0(T)|\geq p^2+p-1+p\bigl(pq-2-(2p-1)\bigr)
 =p^2(q-1)-1.
\]
\end{proof}

\begin{proof}[Proof of Theorem~\ref{thm:1.3}]
Use the factorization $S=RT$ from Lemma~\ref{lem:5.3}. If $\Sigma_0(T)$ contains a coset of $Q$, then $\Sigma_0(\rho(R))=P$ gives $\Sigma_0(S)=G$. Otherwise Lemmas~\ref{lem:5.3} and~\ref{lem:5.4} give
\[
 |\Sigma_0(R)|+|\Sigma_0(T)|
 \geq p^2+p-1+p^2(q-1)-1
 =|G|+p-2>|G|.
\]
Consequently, for every $g\in G$, the intersection of $\Sigma_0(R)$ and $g-\Sigma_0(T)$ is nonempty. Hence $\Sigma_0(S)=G$ in this case as well. Since $0\in\Sigma(S)$ by Lemma~\ref{lem:2.4}, we obtain the desired contradiction.
\end{proof}

We next consider
\[
 G=C_q\oplus C_{pq}\cong V\oplus K,
 \qquad V\cong C_q^2,\quad K\cong C_p,
\]
and write $\rho:G\to V$ for the projection. Lemma~\ref{lem:2.3} gives $c_0(G)\geq q^2+p-2$. Suppose that $S$ is regular and
\[
 |S|=q^2+p-2,\qquad\Sigma(S)\ne G.
\]
Since $D(G)=q+pq-1<|S|$, Lemma~\ref{lem:2.4} again applies.

\begin{lemma}\label{lem:5.5}
There is a factorization $S=RT$ such that
\[
 |R|=2q-1,\qquad |T|=q^2-2q+p-1,
\]
$\Sigma_0(\rho(R))=V$, and every term of $T$ has nonzero $K$-coordinate.
\end{lemma}

\begin{proof}
As in Lemma~\ref{lem:5.3}, $|S_V|\leq2q-2$. For each subgroup $L\leq V$ of order $q$, let $a_L$ count the terms whose projection belongs to $L\setminus\{0\}$. Regularity in $L\oplus K$ gives
\[
 \sum_{M\ne L}a_M\geq q^2+p-2-(pq-1)
 =q(q-p)+p-1\geq2q-2.
\]
The argument for \eqref{eq:5.2}, with $q$ in place of $p$, gives
$\sum_L\min\{a_L,q-1\}\geq2q-2$. We may therefore choose $R_0\mid S$ of length $2q-2$ containing $S_V$ such that $\rho(R_0)$ is regular over $V$.

The complement has length $q^2-2q+p$, all its terms have nonzero $K$-coordinate, and at most $p-1$ lie in $K$. Choose a term $g$ with both coordinates nonzero, and put $R=R_0g$ and $T=SR^{-1}$. Theorem~\ref{thm:2.11} gives $|\Sigma_0(\rho(R_0))|\geq q^2-1$. Its union with its translate by $\rho(g)\ne0$ is all of $V$.
\end{proof}

\begin{lemma}\label{lem:5.6}
For the sequence $T$ in Lemma~\ref{lem:5.5}, the set $\Sigma_0(T)$ contains a coset of $K$.
\end{lemma}

\begin{proof}
Suppose otherwise. Put $B=\Sigma_0(\rho(T))$ and $H=\operatorname{st}_V(B)$. We first show that $B=V$.

If $H=\{0\}$, let $z=|T_K|$. Lemma~\ref{lem:2.6} gives $z\leq p-2$, since $p-1$ nonzero terms of $K$ would have subset sums covering $K$. The nonzero projected terms have subset-sum set $B$ and form a regular sequence by Lemma~\ref{lem:2.9}. Their number is at least
\[
 |T|-z\geq(q-1)^2\geq2q-1,
\]
since $q\geq5$. Theorem~\ref{thm:2.11} then gives $B=V$, contrary to $H=\{0\}$.

If $|H|=q$, write $T=IU$, where $I$ consists of the terms whose projections lie in $H$. Lemma~\ref{lem:2.6} gives $|U|\leq q-2$, while Lemma~\ref{lem:2.4} and the absence of a coset of $K$ give
\[
 |I|+1\leq |\Sigma_0(I)|\leq q(p-1).
\]
Thus $|T|\leq pq-3$, contrary to
\[
 |T|-(pq-3)=q(q-p-2)+p+2>0.
\]
Therefore $H=V$ and $B=V$.

Choose $C\mid T$ as in Lemma~\ref{lem:2.12}, with $\Sigma_0(\rho(C))=V$. If $|C|\leq2q-2$, repeated application of Lemma~\ref{lem:5.1} gives
\[
 |\Sigma_0(T)|\geq q^2+q\bigl(|T|-(2q-2)\bigr).
\]
This exceeds $q^2(p-1)$ by
\[
 q\bigl(q^2-(p+2)q+p+1\bigr)>1,
\]
because $q\geq p+2$. If $|C|=2q-1$ and $\rho(C)$ is regular, remove one term and apply Lemma~\ref{lem:5.2}, with the primes interchanged, to obtain $|\Sigma_0(C)|\geq q^2+q-1$. Adjoining the other terms gives a lower bound only one less than the preceding one. In both cases, $|\Sigma_0(T)|>q^2(p-1)$, which is impossible for a subset of $V\oplus K$ containing no coset of $K$.
\end{proof}

\begin{proof}[Proof of Theorem~\ref{thm:1.4}]
Take $S=RT$ as in Lemma~\ref{lem:5.5}. By Lemma~\ref{lem:5.6}, $\Sigma_0(T)$ contains a coset of $K$. Since $\Sigma_0(R)$ projects onto $V$, we obtain $\Sigma_0(S)=G$. Lemma~\ref{lem:2.4} gives $\Sigma(S)=G$, a contradiction.
\end{proof}

\section{Proof of Theorem~\ref{thm:1.5}}

Let $p$ be an odd prime, and write
\begin{equation}\label{eq:6.1}
 H=C_{p^a}\oplus C_{p^bm},
 \qquad 2\leq a\leq b,\quad (m,p)=1,
\end{equation}
where every prime divisor of $m$ is larger than $p$. Let
\[
 e=\exp(H)=p^bm,\qquad d=|H|/p=p^{a-1}e.
\]
In particular, $e\geq p^a\geq p^2$ and $d\geq p^3$.

\begin{lemma}\label{lem:6.1}
Let $q>p$ be prime with $q\nmid pm$, and let $J<H$. Then
\[
 c_0(J\oplus C_q)+c_0(H/J)-2<d(q-1).
\]
If $s\geq2$, $Q=q^s$, and $\alpha=q^{s-1}(q-p)$, then
\[
 c_0(J\oplus C_Q)+c_0(H/J)-2<d(\alpha-1)+1.
\]
\end{lemma}

\begin{proof}
let $j=|J|$. Since $J$ is proper, $j\leq d$, and Lemma~\ref{lem:2.8} gives $c_0(H/J)\leq pd/j$. If $J$ is cyclic, then $j\leq e$ and, for $Q=q^s$ with $s\geq1$,
\[
 c_0(J\oplus C_Q)+c_0(H/J)-2
 \leq Qj+\frac{pd}{j}-2,
 \qquad 1\leq j\leq e.
\]
If $J$ is noncyclic, its Sylow $p$-subgroup has rank two, since all other Sylow subgroups of $H$ are cyclic. Hence $J\oplus C_Q$ has a quotient isomorphic to $C_p^2$, and $p$ is its smallest prime divisor. Theorem~\ref{thm:3.4} gives
\[
 c_0(J\oplus C_Q)+c_0(H/J)-2
 \leq \frac{Qj}{p}+\frac{pd}{j}+p^2-p-3,
 \qquad p^2\leq j\leq d.
\]
Both expressions are convex functions of $j$. Thus it suffices to check $j=1,e$ in the cyclic case and $j=p^2,d$ in the noncyclic case. Using $q\geq p+2$, $a\geq2$, $e\geq p^a$, and $d=p^{a-1}e$, it is not hard to verify that all four endpoint values are less than $d(q-1)$ when $s=1$, and less than $d(\alpha-1)+1$ when $s\geq2$.
\end{proof}

\begin{lemma}\label{lem:6.2}
Let $H$ be as in \eqref{eq:6.1}, and suppose that $c_0(H)=d+p-2$. If $q>p$ is prime, $q\nmid pm$, and $s\geq1$, then
\[
 c_0(H\oplus C_{q^s})=dq^s+p-2.
\]
\end{lemma}

\begin{proof}
Put $Q=q^s$, $G=H\oplus C_Q$, and $c=d+p-2$. Lemma~\ref{lem:2.3} gives the lower bound. Suppose that $S$ is regular and
\[
 |S|=dQ+p-2,\qquad \Sigma(S)\ne G.
\]
Since $G\cong C_{p^a}\oplus C_{eQ}$, we have $D(G)=eQ+p^a-1<|S|$. Lemma~\ref{lem:2.4} therefore applies. Let $\pi:G\to H$ be the projection.

First let $W\mid S$ satisfy $\Sigma_0(\pi(W))\ne H$. Put
$J=\operatorname{st}_H(\Sigma_0(\pi(W)))<H$, and write $W=IU$, where $I$ consists of the terms whose $H$-coordinates lie in $J$. The sequence $I$ is regular over $J\oplus C_Q$. If $|I|\geq c_0(J\oplus C_Q)$, then $\Sigma_0(I)=J\oplus C_Q$, contrary to Lemma~\ref{lem:2.4}. The image of $U$ in $H/J$ has no zero terms and has a proper subset-sum set with trivial stabilizer. It is regular by Lemma~\ref{lem:2.9}, so $|U|\leq c_0(H/J)-1$. Consequently,
\begin{equation}\label{eq:6.2}
 |W|\leq c_0(J\oplus C_Q)+c_0(H/J)-2.
\end{equation}
Applying Lemma~\ref{lem:6.1} with $W=S$ now shows that $\Sigma_0(\pi(S))=H$.

Suppose first that $s=1$. Let $Z=S_H$. Since $Z$ is regular over $H$, the inequality $|Z|\geq c$ would give $\Sigma_0(Z)=H$, contrary to Lemma~\ref{lem:2.4}. Thus $|Z|\leq c-1$. Apply Lemma~\ref{lem:2.10} to $\pi(S)$, prescribing the terms of $Z$, to choose $R\mid S$ such that
\[
 Z\mid R,\qquad |R|\leq c,\qquad \Sigma_0(\pi(R))=H.
\]
Put $T=SR^{-1}$. Every term of $T$ has nonzero $C_q$-coordinate, and $|T|\geq d(q-1)$. By \eqref{eq:6.2} and Lemma~\ref{lem:6.1}, $\Sigma_0(\pi(T))=H$. A further application of Lemma~\ref{lem:2.10}, with the empty sequence prescribed, gives $C\mid T$ such that $|C|\leq c$ and $\Sigma_0(\pi(C))=H$.

Suppose that $\Sigma_0(T)$ contains no coset of $C_q$. Starting with $\Sigma_0(C)$, adjoin the remaining terms of $T$. Lemma~\ref{lem:5.1}, together with $|H|/\exp(H)=p^a$, gives
\[
 |\Sigma_0(T)|\geq |H|+p^a(|T|-c)
 \geq |H|+p^a\bigl(d(q-2)-p+2\bigr).
\]
This exceeds $|H|(q-1)$ by
\[
 p^a\bigl((q-2)(d-e)-p+2\bigr)>0,
\]
since $d-e=(p^{a-1}-1)e\geq(p-1)p^2$. This contradicts the absence of a coset of $C_q$. Hence $\Sigma_0(T)$ contains such a coset. Since $\Sigma_0(R)$ projects onto $H$, we get $\Sigma_0(S)=G$.

Now suppose that $s\geq2$, and put $\alpha=q^{s-1}(q-p)$. Choose $R\mid S$ with $|R|\leq c$ and $\Sigma_0(\pi(R))=H$, and put $T=SR^{-1}$. The terms of $S$ whose $C_Q$-coordinates do not generate $C_Q$ lie in the proper subgroup $H\oplus qC_Q$, of order $pdq^{s-1}$. By regularity, at least
\[
 dQ+p-2-(pdq^{s-1}-1)=d\alpha+p-1
\]
terms lie outside this subgroup. Let $U\mid T$ consist of the terms whose $C_Q$-coordinates generate $C_Q$. Then
\[
 |U|\geq d\alpha+p-1-c=d(\alpha-1)+1.
\]
It follows from \eqref{eq:6.2} and Lemma~\ref{lem:6.1} that $\Sigma_0(\pi(U))=H$. Choose $C\mid U$ with $|C|\leq c$ and $\Sigma_0(\pi(C))=H$.

If $\Sigma_0(U)$ contains no coset of $C_Q$, Lemma~\ref{lem:5.1} gives
\[
 |\Sigma_0(U)|\geq |H|+p^a(|U|-c)
 \geq |H|+p^a\bigl(d(\alpha-2)-p+3\bigr).
\]
Subtracting $|H|(Q-1)$ from the last expression yields
\begin{equation}\label{eq:6.3}
 p^a\bigl(e[2-Q+p^{a-1}(\alpha-2)]-p+3\bigr).
\end{equation}
This is positive, since $e\geq p^2$ and
\begin{align*}
 2-Q+p^{a-1}(\alpha-2)
 &\geq q^{s-1}\bigl(p(q-p)-q\bigr)-2p+2\\
 &\geq p^2-2p-2\geq1.
\end{align*}
Here the second inequality follows by taking the smallest possible values $s=2$ and $q=p+2$; the expression increases with $s$ and $q$ in the stated ranges. Thus $|\Sigma_0(U)|>|H|(Q-1)$, a contradiction. Hence $\Sigma_0(U)$, and therefore $\Sigma_0(T)$, contains a coset of $C_Q$. Again $\Sigma_0(\pi(R))=H$ gives $\Sigma_0(S)=G$.

In both cases, Lemma~\ref{lem:2.4} gives $0\in\Sigma(S)$ and hence $\Sigma(S)=G$, which is the desired contradiction.
\end{proof}

\begin{proof}[Proof of Theorem~\ref{thm:1.5}]
Let $n=p^tm$ with $(m,p)=1$, and absorb $p^t$ into $p^b$. It is therefore enough to consider $(n,p)=1$. Factor
\[
 n=\prod_{i=1}^u q_i^{s_i},
\]
where the $q_i>p$ are distinct primes. For $n=1$, the result is known for rank-two $p$-groups; see \cite[Theorem~1.1(5)]{GHHQZ}. Starting with $C_{p^a}\oplus C_{p^b}$, apply Lemma~\ref{lem:6.2} successively to the factors $q_i^{s_i}$. Each new prime is coprime to the group already obtained, so the hypotheses remain satisfied. The desired result follows easily.
\end{proof}

\begin{proof}[Proof of Corollary~\ref{cor:1.6}]
The result is known for cyclic groups, groups of even order, and finite abelian $p$-groups; see \cite{GHHQZ,QuHan2017}. Theorem~\ref{thm:1.2} treats all groups of rank at least three. It remains to consider
\[
 G=C_{n_1}\oplus C_{n_2},\qquad 1<n_1\mid n_2,
\]
where $|G|$ is odd and $G$ is not a $p$-group.

If $n_1$ has a prime divisor $q>p$, then $G$ has a quotient isomorphic to $C_q^2$. Write $|G|=pq^2h$. If $h>1$, then $h\geq p$, and Theorem~\ref{thm:3.3} applies. If $h=1$, the invariant factors give $G\cong C_q\oplus C_{pq}$, which is covered by Theorem~\ref{thm:1.4}.

Otherwise $n_1=p^a$. The case $a=1$ is precisely the family excluded in the statement. If $a\geq2$, write $n_2=p^bn$ with $b\geq a$ and no prime divisor of $n$ smaller than $p$. The conclusion follows from Theorem~\ref{thm:1.5}.
\end{proof}

\section{The groups \texorpdfstring{$C_p\oplus C_{pn}$}{Cp + Cpn}}\label{sec:7}

Let $G=C_p\oplus C_{pn}$, where $n\geq 2$ and $p$ is the smallest
prime divisor of $|G|$. Let
\[
 K=pG\cong C_n,\qquad V=G/K\cong C_p^2,
\]
and let $\pi:G\longrightarrow V$ be the natural homomorphism. Throughout
this section, suppose that $n\geq 2p-1$ and that $S$ is a regular sequence
with
\[
 |S|=pn+2p-3,\qquad \Sigma(S)\ne G.
\]
Since $D(G)=pn+p-1\leq |S|$, Lemma~\ref{lem:2.4} applies.
For each subgroup $M$ of $V$ of order $p$, let $H_M=\pi^{-1}(M)$.
Thus $|H_M|=pn$, and there are $p+1$ such subgroups.

\begin{lemma}\label{lem:7.1}
For each subgroup $M\leq V$ of order $p$, one can choose a subsequence
$T_M\mid S$ of length $p-1$ whose terms lie outside $H_M$, so that the
$p+1$ sequences $T_M$ are pairwise disjoint. After any such choice,
distribute the remaining terms into sequences $A_M$ over $H_M$: each
term outside $K$ has a unique possible assignment, while the terms in
$K$ may be assigned arbitrarily. Then
\[
 \sum_M e(A_M)\leq (p-1)(p-2)-1.
\]
\end{lemma}

\begin{proof}
For each $M$, regularity gives at least
\[
 |S|-|S_{H_M}|\geq pn+2p-3-(pn-1)=2p-2
\]
terms outside $H_M$. For two or more distinct subgroups $M$, the union
of the available terms consists of all terms outside $K$. Their number
is at least
\[
 |S|-|S_K|\geq pn+2p-3-(n-1)
 =(p-1)n+2p-2\geq p^2-1.
\]
Lemma~\ref{lem:3.1} therefore gives the required disjoint subsequences.

Write $E=\sum_M e(A_M)$. Since $p^2-1$ terms have been removed,
\[
 \sum_M |A_M|=pn-p^2+2p-2.
\]
The image of $T_M$ in $G/H_M\cong C_p$ has subset sums covering that
quotient. Hence, by Lemma~\ref{lem:2.7},
\[
 |\Sigma_0(A_MT_M)|\geq p|\Sigma_0(A_M)|,
\]
and consequently
\[
 |\Sigma_0(A_MT_M)|-1
 \geq |A_MT_M|+(p-1)|A_M|+pe(A_M).
\]
Applying Theorem \ref{thm:2.1} to the $p+1$ disjoint factors $A_MT_M$ of $S$,
we obtain
\begin{align*}
 |\Sigma_0(S)|
 &\geq 1+|S|+(p-1)\sum_M|A_M|+pE\\
 &=|G|-p(p-1)(p-2)+pE.
\end{align*}
Since $\Sigma_0(S)\ne G$ and $E$ is an integer, the assertion follows.
\end{proof}
\begin{lemma}\label{lem:7.2}
Let $Z=S_K$ and $z=|Z|$. Then
\[
 (p-1)z+pe(Z)\leq (p-1)(p-2)-1.
\]
In particular, $z\leq p-3$. Moreover, $z=p-3$ implies $e(Z)=0$, and
$z=p-4$ implies $e(Z)\leq 1$.
\end{lemma}

\begin{proof}
No term of any $T_M$ lies in $K$. After the sequences $T_M$ have been
removed, at least
\[
 pn+2p-3-(n-1)-(p^2-1)
 =(p-1)n-p^2+2p-1\geq p(p-1)
\]
terms remain outside $K$. Since $p(p-1)>(p+1)(p-2)$, some $H_M$
contains at least $p-1$ of these terms. Choose $p-1$ of them to form a
sequence $Y$, and assign all of $Z$ to the same $A_M$.
The image of $Y$ in $H_M/K\cong C_p$ covers that quotient. Therefore
\[
 |\Sigma_0(ZY)|\geq p|\Sigma_0(Z)|=p(z+1+e(Z)),
\]
which gives
\[
 e(ZY)\geq (p-1)z+pe(Z).
\]
Now $e(A_M)\geq e(ZY)$ by Lemma~\ref{lem:2.5}, so the required
inequality follows from Lemma~\ref{lem:7.1}. Its remaining assertions
follow by substitution.
\end{proof}

Let $P_0,\ldots,P_p$ be the subsequences of $S$ whose images lie in the
$p+1$ respective sets $M\setminus\{0\}$. The next lemma allows us to
specify the terms removed in Lemma~\ref{lem:7.1}.

\begin{lemma}\label{lem:7.3}
Suppose that $R_i\mid P_i$ satisfy
\[
 \sum_{i=0}^p|R_i|=p^2-1,\qquad
 |R_i|\leq p(p-1)\quad(0\leq i\leq p).
\]
Then the terms of $R_0\cdots R_p$ can be distributed among the
sequences $T_M$ in Lemma~\ref{lem:7.1}.
\end{lemma}

\begin{proof}
Apply Hall's theorem as in Lemma~\ref{lem:3.1}. There are $p-1$ places
for each of the $p+1$ subgroups. A term from $P_i$ can occupy every
place except those for its own subgroup. Thus a collection of terms
from one $P_i$ has $p(p-1)$ available places, whereas a collection
involving at least two different $P_i$ has all $p^2-1$ places available.
The stated inequalities verify Hall's condition.
\end{proof}

\section{Proof of Theorem~\ref{thm:1.7}}\label{sec:8}

Let $n\geq 5$ be odd, let $G=C_3\oplus C_{3n}$, and let
$K=3G\cong C_n$ and $V=G/K\cong C_3^2$.
Suppose that $S$ is regular, $|S|=3n+3$, and $\Sigma(S)\ne G$.
By Lemma~\ref{lem:7.2}, $S_K=1$, so we may write
\[
 S=P_0P_1P_2P_3,
\]
where $P_i$ consists of the terms whose images lie in the $i$th
subgroup of $V$ of order three, with zero omitted. Let $H_i$ be the
inverse image of that subgroup, so that $|H_i|=3n$.

By Lemmas~\ref{lem:7.1} and~\ref{lem:7.3}, every choice of $R_i\mid P_i$
with
\begin{equation}\label{eq:8.1}
 \sum_{i=0}^3|R_i|=8,\qquad |R_i|\leq 6
\end{equation}
satisfies
\begin{equation}\label{eq:8.2}
 \sum_{i=0}^3 e(P_iR_i^{-1})\leq 1.
\end{equation}
We first establish the elementary facts needed to use this inequality.

\begin{lemma}\label{lem:8.1}
Let $A$ be a sequence of nonzero elements of an abelian group of odd
order, and suppose that $\operatorname{st}(\Sigma_0(A))=\{0\}$.
\begin{enumerate}
\item[\textup{(i)}] If $e(A)\geq 2$, then some $B\mid A$ satisfies
 $|B|\leq 3$ and $e(B)\geq 2$.
\item[\textup{(ii)}] If $e(A)=0$, then $\Sigma_0(A)$ is an arithmetic
 progression of size $|A|+1$.
\item[\textup{(iii)}] If $e(A)=1$ and $|A|\geq 3$, then $\Sigma_0(A)$
 is an arithmetic progression of size $|A|+2$.
\end{enumerate}
\end{lemma}

\begin{proof}
Every subsequence of $A$ also has a subset-sum set with trivial
stabilizer. Thus $e(B)\leq e(A)$ for $B\mid A$, by
Lemma~\ref{lem:2.5}. The assertions for $|A|\leq 2$ are immediate.

Suppose that $|A|\geq 3$ and that $e(B)\leq 1$ whenever $B\mid A$
and $|B|\leq 3$. If every two-term subsequence has value zero, then
any two terms $x,y$ satisfy
\[
 |\{0,x,y,x+y\}|=3.
\]
Hence $y=x$ or $y=-x$, and all terms of $A$ belong to
$\{g,-g\}$ for some nonzero $g$. Their subset sums form an interval
in $\langle g\rangle$. This interval cannot fill $\langle g\rangle$,
since its stabilizer is trivial, and therefore it has exactly $|A|+1$
elements.

Otherwise choose two terms $x,y$ for which
\[
 Q=\{0,x,y,x+y\}
\]
has four elements. For every other term $z$, our assumption gives
\[
 |Q\cup(Q+z)|=5,\qquad |Q\cap(Q+z)|=3.
\]
The three edges $u\longrightarrow u+z$ within $Q$ form either one
path of length three, or a cycle of length three and an isolated point.
The latter possibility would put three points of $Q$ in one coset of
a subgroup of order three and the fourth point in a different coset.
This is incompatible with the equality of the sums of the opposite
pairs of $Q$. Consequently
\[
 Q=a+\{0,d,2d,3d\}
\]
for some $d$ of odd order $N\geq 5$. Since $0\in Q$, this progression
lies in $\langle d\rangle$. The case $N=5$ would give
$\Sigma_0(xyz)=\langle d\rangle$, so $N\geq 7$.

It is not hard to verify that a four-term progression in a cyclic group
of odd order at least seven has three points in common with a nonzero
translate only when the translation is $d$ or $-d$. Thus every term
other than $x,y$ belongs to $\{d,-d\}$. The only pairing of the
positions $0,1,2,3$ with equal pair sums modulo $N$ is
$\{0,3\},\{1,2\}$. It follows that one of $x,y$ belongs to
$\{d,-d\}$ and the other to $\{2d,-2d\}$. Therefore all but one
terms of $A$ belong to $\{d,-d\}$, and the remaining term is $2d$ or
$-2d$. The subset sums of the former terms form an interval of at
least three points. Adding the remaining term gives another interval,
of at most $|A|+2$ points. In particular, $e(A)\leq 1$, proving
\textup{(i)} by contraposition.

If $e(A)=0$, monotonicity implies that every pair has value zero, so
the case where every pair has value zero proves \textup{(ii)}. If $e(A)=1$ and $|A|\geq 3$,
some pair has value one; otherwise the first part of the argument would give $e(A)=0$.
The preceding argument then shows that $\Sigma_0(A)$ is a progression.
Its size is $|A|+2$ by definition, proving \textup{(iii)}.
\end{proof}

\begin{lemma}\label{lem:8.2}
Let $\varphi:G\longrightarrow C_3$ be surjective, and let $T$ be a
sequence of five terms with nonzero images. If
$\operatorname{st}(\Sigma_0(T))=\{0\}$, then each coset of
$\ker\varphi$ contains at least two elements of $\Sigma_0(T)$.
\end{lemma}

\begin{proof}
Write the two nonzero images as $1$ and $-1$, choosing the sign so that
at least three terms have image $1$. Let $b\in\{0,1,2\}$ be the number
with image $-1$. Every quotient element is represented by a subset sum,
by Lemma~\ref{lem:2.6}. We show that none has only one representative.

Suppose first that the only subset sum with image zero is $0$. If both
signs occur, every term with image $-1$ is the negative of every term
with image $1$. Thus all positive-image terms equal some $g$, and
three such terms give $3g=0$. If all five images are $1$, comparison
of overlapping triples gives the same conclusion.

Suppose next that there is only one subset sum with image $1$. All
terms with that image equal some $g$. If $b=0$, comparing a one-term
sum with a four-term sum gives $g=4g$. If $b=1$, write $h$ for the
remaining term. Comparing $g$ with $2g+h$ gives $h=-g$, and again
$4g=g$. If $b=2$, the same comparison gives $h_1=h_2=-g$; their sum
then gives $-2g=g$.

Finally, suppose that there is only one subset sum with image $-1$.
Comparison of the sums of two positive-image terms shows that all
positive-image terms equal some $g$. Every negative-image term equals
$2g$. If $b=0$, compare $2g$ with $5g$. If $b=1$, compare $h$ with
$3g+h$. If $b=2$, compare $h_1$ with $g+h_1+h_2$.

In every case, all terms belong to $\{g,-g\}$ with $3g=0$.
Since $\varphi(g)\ne 0$, this gives
$\Sigma_0(T)=\langle g\rangle$ with nontrivial stabilizer, a
contradiction.
\end{proof}

\begin{lemma}\label{lem:8.3}
For the four subsequences $P_i$ defined above,
\[
 \sum_{i=0}^3e(P_i)\leq 1.
\]
\end{lemma}

\begin{proof}
Suppose that $e(P_i)\geq 2$ for some $i$. By Lemma~\ref{lem:8.1},
there is a subsequence $B\mid P_i$ of length three with $e(B)\geq 2$.
If eight terms satisfying~\eqref{eq:8.1} can be chosen outside $B$,
then $B$ remains in $P_iR_i^{-1}$, contradicting~\eqref{eq:8.2}.

Otherwise, let $a_j$ be the number of terms remaining in $P_j$ after
$B$ is removed. Then
\[
 \sum_j a_j=3n,\qquad \sum_j\min\{6,a_j\}\leq 7.
\]
Some $a_j$ is at least six. All other $a_k$ together are then at most
one, and hence $a_j\geq 3n-1$. It follows that
\[
 |\Sigma_0(P_j)|\geq |P_j|+1\geq 3n=|H_j|,
\]
so $\Sigma_0(P_j)=H_j$, contrary to Lemma~\ref{lem:2.4}.
Thus $e(P_i)\leq 1$ for every $i$.

If two different $P_i$ have value one, each contains a two-term
subsequence with value one, by the proof of Lemma~\ref{lem:8.1}.
Keep those four terms. If eight terms satisfying~\eqref{eq:8.1} can
be removed from the rest, the two pairs contradict~\eqref{eq:8.2}.
If not, the preceding calculation, with residual total $3n-1$, gives
some $P_j$ of length at least $3n-2$. Since
$\Sigma_0(P_j)\subseteq H_j$ has trivial stabilizer, we must have
\[
 |P_j|=3n-2,\qquad e(P_j)=0,\qquad
 \Sigma_0(P_j)=H_j\setminus\{h\}
\]
for some $h\in H_j$. The remaining five terms of $S$ lie outside
$H_j$. By Lemma~\ref{lem:8.2}, their subset sums have two distinct
elements in each $H_j$-coset. The corresponding two translates of
$H_j\setminus\{h\}$ fill that coset. Hence $\Sigma_0(S)=G$, again
a contradiction.
\end{proof}

\begin{lemma}\label{lem:8.4}
Let $N$ be odd, let $2\leq k\leq 4$, and let $A_1,\ldots,A_k$ be
nonempty arithmetic progressions in $C_N$. Suppose that the differences
of those progressions having at least two elements generate $C_N$.
If
\[
 \sum_{i=1}^k|A_i|\geq N+k-1,
\]
then $A_1+\cdots+A_k=C_N$.
\end{lemma}

\begin{proof}
Suppose that the sum is proper, and let
$H=\operatorname{st}(A_1+\cdots+A_k)$. If $H=\{0\}$, Theorem \ref{thm:2.1}
immediately gives a contradiction. Otherwise, writing
\[
 h=|H|,\qquad t=[C_N:H],
\]
we have $h,t\geq 3$.

Let $r$ be the number of progressions having at least two elements
whose differences do not belong to $H$. The sum of the images in
$C_N/H$ has trivial stabilizer. For each of these $r$ progressions,
its size is smaller than the order of its difference in the quotient.
Indeed, otherwise its image would be a coset of a nontrivial subgroup,
which would stabilize the quotient sumset. Its image therefore has
the same size as the progression itself. Each of the remaining
$s=k-r$ progressions has a singleton image and contains at most $h$
elements.

Theorem \ref{thm:2.1} shows that 
\[
 \sum_{\substack{i\text{ among the}\\r\text{ progressions}}}|A_i|
 \leq t+r-2.
\]
The generating assumption gives $r\geq 1$, so $s\leq k-1\leq 3\leq t$.
Consequently
\[
 \sum_{i=1}^k|A_i|\leq t+r-2+sh<N+k-1,
\]
because the difference between the last two bounds is
\[
 N+k-1-(t+r-2+sh)=(h-1)(t-s)+1>0.
\]
This is the required contradiction.
\end{proof}

By Lemmas~\ref{lem:8.1} and~\ref{lem:8.3}, every $\Sigma_0(P_i)$ is
an arithmetic progression, except possibly when $|P_i|=2$ and
$e(P_i)=1$. When all four sumsets are progressions, write
\begin{equation}\label{eq:8.3}
 \Sigma_0(P_i)=a_i+\{0,g_i,\ldots,\ell_i g_i\},\qquad
 \ell_i=|P_i|+e(P_i).
\end{equation}
For an empty $P_i$, take $a_i=0$, $\ell_i=0$, and choose any $g_i$
whose image is nonzero in the assigned subgroup of $V$.
For a nonempty $P_i$, consecutive points of the progression show that
$g_i\in H_i$. Since the progression contains zero and a term outside
$K$, its difference is not in $K$. Thus $g_i+K$ is nonzero in the
assigned subgroup in every case. Moreover, the trivial stabilizer gives
\begin{equation}\label{eq:8.4}
 \ell_i\leq \operatorname{ord}(g_i)-2.
\end{equation}
Put $d_i=3g_i\in K$. Since $g_i+K$ has order three,
\[
 \operatorname{ord}(g_i)=3\operatorname{ord}(d_i).
\]

\begin{lemma}\label{lem:8.5}
If all four sumsets have the form~\eqref{eq:8.3}, then the elements
$d_i$ with $\ell_i\geq 5$ generate $K$.
\end{lemma}

\begin{proof}
Let $D$ be the subgroup they generate, and suppose that $D<K$.
If $[K:D]\ne 3$, then $[K:D]\geq 5$. For every index with
$\ell_i\geq 5$, inequality~\eqref{eq:8.4} gives
$\ell_i\leq 3n/5-2$; all other $\ell_i$ are at most four. Hence
\[
 \sum_{i=0}^3\ell_i\leq\max\{16,12n/5-8\}<3n+3,
\]
contrary to $\sum_i\ell_i=|S|+\sum_i e(P_i)\geq 3n+3$.

Suppose that $[K:D]=3$. Then $3\mid n$, $n\geq 9$, and $D=3K$.
Write
\[
 G=\langle e\rangle\oplus\langle f\rangle,\qquad
 \operatorname{ord}(e)=3,\quad \operatorname{ord}(f)=3n.
\]
For $g=ae+bf$, the condition $3g\in 3K$ is equivalent to
$b\equiv 0\pmod 3$. Thus every $g_i$ with $\ell_i\geq 5$ has image
in the subgroup generated by $e+K$. There is at most one such index,
and its parameter is at most $3|D|-2=n-2$. Therefore
\[
 \sum_{i=0}^3\ell_i\leq n-2+12=n+10<3n+3,
\]
again a contradiction.
\end{proof}

\begin{proof}[Proof of Theorem~\ref{thm:1.7}]
The lower bound follows from Lemma~\ref{lem:2.3}. The case $n=3$
follows from the known result for finite abelian $3$-groups
\cite{QuHan2017}. Let $n\geq 5$ be odd, and suppose that a regular
sequence $S$ of length $3n+3$ is not an additive basis. The preceding
results apply.

First suppose that all four sumsets have the form~\eqref{eq:8.3}, and put
\[
 \varepsilon=\sum_{i=0}^3e(P_i)\in\{0,1\},\qquad
 \sum_{i=0}^3\ell_i=3n+3+\varepsilon.
\]
Every $\ell_i$ is at most $3n-2$, by~\eqref{eq:8.4}, so at least two
of them are at least two.

Fix a target element of $V$, and subtract the image of $\sum_i a_i$
to obtain $x\in V$. Write $v_i=g_i+K$. For $r\in\{0,1,2\}$, let
$s_i(r)$ be the number of integers in $[0,\ell_i]$ congruent to $r$
modulo three. The corresponding part of the progression, when
nonempty, is a translate of a progression in $K$ with difference $d_i$
and size $s_i(r)$, and
\[
 s_i(0)+s_i(1)+s_i(2)=\ell_i+1.
\]
If $\ell_i\geq 2$, all three parts are nonempty. If $\ell_i\leq 1$,
every nonempty part is a singleton. If $\ell_i\geq 5$, every part has
at least two elements. Thus, whenever parts from all indices with
$\ell_i\geq 5$ are selected, their nonsingleton differences generate
$K$ by Lemma~\ref{lem:8.5}.

Suppose first that all four $\ell_i$ are at least two. The equation
\[
 r_0v_0+r_1v_1+r_2v_2+r_3v_3=x,\qquad r_i\in\{0,1,2\},
\]
has nine solutions. Each coefficient takes each value exactly three
times, since any two distinct $v_i$ generate $V$. The average of
$\sum_i s_i(r_i)$ over these solutions is
\[
 \frac13\sum_{i=0}^3(\ell_i+1)
 =\frac{3n+7+\varepsilon}{3}>n+2.
\]
Some solution therefore selects four nonempty progressions whose total
size is at least $n+3$. Lemma~\ref{lem:8.4} shows that their sum fills
the chosen coset of $K$.

Suppose next that exactly three $\ell_i$ are at least two. Fix a
nonempty part of the remaining progression, whose parameter is
$t\leq 1$. The coefficient equation in the other three directions has
three solutions, and each coefficient takes each value once. The
average total size of these three selected parts is
\[
 \frac{3n+6+\varepsilon-t}{3}>n+1.
\]
One solution therefore has total size at least $n+2$. All indices with
$\ell_i\geq 5$ occur among these three, so Lemma~\ref{lem:8.4} again
fills the chosen coset.

It remains to consider exactly two indices with $\ell_i\geq 2$, say
$i=0,1$. Put
\[
 m_i=\ell_i+1,\qquad q_i=\left\lfloor\frac{m_i}{3}\right\rfloor
 \quad(i=0,1),\qquad t=\ell_2+\ell_3\leq 2.
\]
Then
\[
 m_0+m_1=3n+5+\varepsilon-t\geq 3n+3,
 \qquad q_0+q_1\geq n.
\]
After nonempty parts of the two small progressions are fixed, the two
coefficients $r_0,r_1$ are uniquely determined. If $q_0+q_1\geq n+1$,
the selected parts already have total size at least $n+1$.

Suppose that $q_0+q_1=n$. The remainders $m_i-3q_i$ have sum at
least three, so one of them is two. Also $t\geq 1$, since otherwise
their sum would be at least five, whereas each remainder is at most
two. In the progression with remainder two, two parts have size
$q_i+1$ and only one has size $q_i$. Choose $j\in\{2,3\}$ with
$\ell_j=1$, and vary its coefficient between zero and one while
keeping the other small coefficient fixed. Both $r_0$ and $r_1$
change, since $v_j$ belongs to neither of their subgroups. At least
one choice therefore selects a part of size $q_i+1$ in the progression
with remainder two. The two large parts then have total size at least
$n+1$.

In both alternatives the two selected parts are nonempty, and their
nonsingleton differences include all $d_i$ with $\ell_i\geq 5$.
Lemma~\ref{lem:8.4} fills the chosen coset. Since the target was
arbitrary, this proves $\Sigma_0(S)=G$ when all four sumsets are
progressions.

Finally, suppose that one $P_i$ consists of two terms, has $e(P_i)=1$,
and its sumset is not a progression. The other three sequences have
value zero and total length $3n+1$. Write their sumsets in the
form~\eqref{eq:8.3}; for them $\ell_i=|P_i|$.
Their elements $d_i$ with $\ell_i\geq 5$ generate $K$. Indeed, if the
generated subgroup had index at least five, the argument of
Lemma~\ref{lem:8.5} would give
\[
 \sum_i\ell_i\leq\max\{12,9n/5-6\}<3n+1.
\]
If its index were three, at most one of the three parameters could be
at least five, giving
\[
 \sum_i\ell_i\leq n-2+8=n+6<3n+1.
\]
Both are impossible.

At least two of these three parameters are at least two; otherwise
their total would be at most $(3n-2)+2=3n$. If all three are at least
two, choose the empty subsequence from the exceptional sequence.
For each target in $V$, the other three coefficients have three
solutions and each takes each value once. Their average total size is
\[
 \frac{(3n+1)+3}{3}=n+\frac43.
\]
Some solution has total size at least $n+2$, so Lemma~\ref{lem:8.4}
fills the target coset.

If just two progression parameters are at least two, fix a nonempty
part of the third, whose parameter is $t\leq 1$. The projected subset
sums of the exceptional two-term sequence cover its subgroup of order
three. Choose one actual subset sum for each of those three images.
As the image varies, the uniquely determined coefficients in the two
remaining directions each take every value once. The average total
size of their parts is
\[
 \frac{3n+3-t}{3}>n.
\]
Thus one choice gives two nonempty progressions of total size at least
$n+1$. Their nonsingleton differences include all the generating $d_i$
with $\ell_i\geq 5$. Lemma~\ref{lem:8.4} again fills the target coset.
This proves $\Sigma_0(S)=G$ in the last case, completing the desired
contradiction.
\end{proof}

\begin{proof}[Proof of Corollary~\ref{cor:1.8}]
By Corollary~\ref{cor:1.6}, it remains to consider
$G=C_3\oplus C_{3n}$, where $n$ is odd. For $n=1$, the result is
Theorem \ref{thm:2.11}; for $n\geq 3$, it follows from
Theorem~\ref{thm:1.7}.
\end{proof}

\medskip
\noindent\textbf{Statement of AI.}
ChatGPT was used to assist with routine calculations and checks, language
editing, and reorganization of the manuscript. The authors are completely responsible for the proofs and statements.

\medskip
\noindent\textbf{Acknowledgment.} This work was supported by Guangdong
Basic and Applied Basic Research Foundation Grant No.~2024A1515012564.


\begin{thebibliography}{99}

\bibitem{Davenport1935}
H.~Davenport, \href{https://doi.org/10.1112/jlms/s1-10.37.30}{\emph{On the addition of residue classes}},
J. London Math. Soc. \textbf{10} (1935), 30--32.


\bibitem{FreezeGaoGeroldinger}
M.~Freeze, W.~Gao, and A.~Geroldinger,
\emph{The critical number of finite abelian groups},
J. Number Theory \textbf{129} (2009), 2766--2777;
Corrigendum, ibid. \textbf{152} (2015), 205--207.

\bibitem{GaoHamidoune}
W.~Gao and Y.~O.~Hamidoune, \emph{On additive bases},
Acta Arith. \textbf{88} (1999), 233--237.

\bibitem{GHHQZ}
W.~Gao, D.~Han, G.~Qian, Y.~Qu, and H.~Zhang,
\href{https://doi.org/10.4064/aa168-3-3}{\emph{On additive bases II}},
Acta Arith. \textbf{168} (2015), 247--267.


\bibitem{GLQW}
W.~Gao, Y.~Li, Y.~Qu, and Q.~Wang,
\emph{Additive bases of finite abelian groups of rank $2$},
Integers \textbf{22} (2022), Paper A37, 15 pp.

\bibitem{GPZ2013}
W.~Gao, J.~Peng, and Q.~Zhong,
\href{https://doi.org/10.4064/aa158-3-6}{\emph{A quantitative aspect of non-unique factorizations: the Narkiewicz constants III}},
Acta Arith. \textbf{158} (2013), 271--285.


\bibitem{GQZ}
W.~Gao, Y.~Qu, and H.~Zhang,
\href{https://doi.org/10.4064/aa181106-22-4}{\emph{On additive bases III}},
Acta Arith. \textbf{193} (2020), 293--308.

\bibitem{GeroldingerHalterKoch2006}
		A.~Geroldinger and F.~Halter-Koch,
		\emph{Non-Unique Factorizations: Algebraic, Combinatorial and Analytic Theory},
		Pure and Applied Mathematics, vol.~278, Chapman \& Hall/CRC, Boca Raton, 2006.

\bibitem{Grynkiewicz}
D.~J.~Grynkiewicz, \emph{Structural Additive Theory},
Developments in Mathematics 30, Springer, New York, 2013.

\bibitem{Hall}
P.~Hall, \href{https://doi.org/10.1112/jlms/s1-10.37.26}{\emph{On representatives of subsets}},
J. London Math. Soc. \textbf{10} (1935), 26--30.

\bibitem{HanRen}
D.~Han and Y.~Ren, \emph{When Do Restricted Sumsets Cover a Finite Abelian Group?} \href{https://arxiv.org/abs/2610.14496v1}{arXiv:2610.14496v1} (2026).

\bibitem{HanZhang}
D.~Han and H.~Zhang, \emph{Davenport Constants of Finite Abelian Groups}, \href{https://arxiv.org/abs/2610.15102v1}{arXiv:2610.15102v1} (2026).


\bibitem{Kneser}
M.~Kneser, \href{https://doi.org/10.1007/BF01174162}{\emph{Absch\"atzung der asymptotischen Dichte von Summenmengen}},
Math. Z. \textbf{58} (1953), 459--484.


\bibitem{Nathanson}
M.~B.~Nathanson,
\emph{Additive Number Theory: Inverse Problems and the Geometry of Sumsets},
Graduate Texts in Mathematics 165, Springer, New York, 1996.

\bibitem{PengI}
C.~Peng, \emph{Addition theorems in elementary abelian groups I},
J. Number Theory \textbf{27} (1987), 46--57.

\bibitem{PengII}
C.~Peng, \emph{Addition theorems in elementary abelian groups II},
J. Number Theory \textbf{27} (1987), 58--62.

\bibitem{PengQuLi}
J.~Peng, Y.~Qu, and Y.~Li,
\href{https://doi.org/10.1007/s11464-020-0869-2}{\emph{Inverse problems associated with subsequence sums in $C_p\oplus C_p$}},
Front. Math. China \textbf{15} (2020), 985--1000.


\bibitem{QuGao}
Y.~Qu and W.~Gao, \href{https://doi.org/10.4064/aa230323-10-4}{\emph{On additive bases of finite groups}},
Acta Arith. \textbf{215} (2024), 97--113.


\bibitem{QuHan2016}
Y.~Qu and D.~Han, \emph{An inverse theorem for additive bases},
Int. J. Number Theory \textbf{12} (2016), 1509--1518.

\bibitem{QuHan2017}
Y.~Qu and D.~Han, \href{https://doi.org/10.1142/S1793042117501342}{\emph{Additive bases of $C_p\oplus C_{p^n}$}},
Int. J. Number Theory \textbf{13} (2017), 2453--2459.


\bibitem{QuLi}
Y.~Qu and Y.~Li, \href{https://doi.org/10.4064/cm8515-6-2021}{\emph{Additive bases of $C_3\oplus C_{3q}$}},
Colloq. Math. \textbf{169} (2022), 189--195.


\bibitem{QLWZ}
Y.~Qu, Y.~Li, Q.~Wang, and X.~Zhao,
\href{https://doi.org/10.1007/s00373-025-02895-9}{\emph{A note on additive bases of abelian groups of rank $2$}},
Graphs Combin. \textbf{41} (2025), Article 24.


\bibitem{Yang}
W.~Yang, \emph{Additive bases and the zero-sum invariant $\mathsf N(G)$ for groups $C_p\oplus C_{np}$},
\href{https://arxiv.org/abs/2610.13103v1}{arXiv:2610.13103v1} (2026).

\end{thebibliography}
\end{document}